\documentclass{siamltex}
\usepackage{amsmath,amsfonts,graphicx,fancybox,color}
\usepackage{MnSymbol}

\newcommand{\RR}{{\mathbb{R}}}
\newcommand{\CC}{{\mathbb{C}}}
\newcommand{\NN}{{\mathbb{N}}}

\newtheorem{remark}[theorem]{Remark}
\newtheorem{example}[theorem]{Example}

\begin{document}
\bibliographystyle{siam}

\pagestyle{myheadings}
\markboth{V.~Simoncini}{Adaptive rational Krylov method for Riccati equations}

\title{Analysis of the rational Krylov subspace projection method
 for large-scale algebraic Riccati equations 
\thanks{Version of January 31, 2016.}}
\author{
V. Simoncini\thanks{Dipartimento di Matematica,
Universit\`a di Bologna,
Piazza di Porta S. Donato  5, I-40127 Bologna, Italy and
IMATI-CNR, Pavia, Italy (valeria.simoncini@unibo.it).}
}
\maketitle

\begin{abstract}
In the numerical solution of the algebraic Riccati equation
$A^* X + X A - X BB^* X + C^* C =0$,
where $A$ is large, sparse and stable, and $B$, $C$ have low rank, projection
methods have recently emerged as a possible alternative to the more
established Newton-Kleinman iteration. 
In spite of convincing numerical experiments, a systematic matrix analysis
of this class of methods is still lacking.
We derive new relations for the approximate solution, the residual and the
error matrices, giving new insights into the role of the
matrix $A-BB^*X$ and of its approximations in the numerical procedure.  
The new results provide theoretical ground for recently proposed modifications
of projection methods onto rational Krylov subspaces.
\end{abstract}

\begin{keywords}
Riccati equation, rational Krylov, reduced order modelling
\end{keywords}

\begin{AMS}
47J20, 65F30, 49M99, 49N35, 93B52
\end{AMS}

\section{Introduction}
We consider the numerical solution of the algebraic Riccati equation
\begin{eqnarray}\label{eqn:main}
A^* X + X A - X BB^* X + C^* C =0,
\end{eqnarray}
where $A \in \RR^{n\times n}$ is large and sparse, and $B\in\RR^{n\times q}$, $C\in\RR^{p\times n}$
with $q,p\ll n$; here and in the following $A^*$ denotes the complex conjugate of $A$.
For $A$ stable\footnote{An $n\times n$ matrix is stable if all its eigenvalues are
in the open left half complex plane.}, the solution matrix $X$ of interest 
is the one that is symmetric positive semidefinite and such that 
$A-BB^*X$ remains stable. 
 Equation (\ref{eqn:main}) arises in many scientific and engineering
applications that require controlling a dynamical system, and it has been
deeply studied by applied algebraists and
numerical mathematicians; we refer the reader to \cite{Lancaster.Rodman.95}
for a thorough description of the problem and its many mathematical relations.
In the recent book \cite{Binietal.book.12}, the numerical treatment of this
and related problems has been discussed, both in the small and large scale cases.
In the large scale setting, with $n\gg 10^3$, a serious bottleneck is given
by the fact that the possibly dense $n\times n$ matrix $X$ cannot be stored. Most numerical
methods thus approximate $X$ by means of factored low-rank matrices, e.g., $X\approx ZZ^*$,
so that only $Z$ needs to be stored.
Different approaches have been explored to solve (\ref{eqn:main}) under this
constraint, and for quite some time a variant of the Newton method, the Newton-Kleinman
iteration, has been the most popular approach
\cite{Kleinman_68},\cite{FHS.09},\cite{Benner.Saak.10},\cite{Benner.Li.Penzl.08}.
Low-rank subspace iteration strategies have also been considered in the past few
years, see, e.g., 
\cite{Amodei.Buchot.10},\cite{Benner.Bujanovic.16},\cite{Lin.Simoncini.15}.
Other forms of data-sparse approximations include multilevel 
\cite{Grasedyck.08} and hierarchical \cite{Grasedyck.Hackbusch.Khoromskij.03} methods,
which  rely on available structure in the data.

Projection-type methods also yield low rank approximations, however they had not
been used for the Riccati equation until very recently. In fact,
projection methods are extensively employed in the solution of algebraic
linear systems and eigenvalue problems. In the past decade, specific choices
of approximation spaces have shown that projection methods are
particularly effective also for {\it linear} matrix equations such as
the Lyapunov and Sylvester equations \cite{Simoncini.survey13}. Lately,
the projection idea has been applied to the algebraic (quadratic)
Riccati equation \cite{Jbilou_03},\cite{Heyouni.Jbilou.09},
 with surprisingly good results, to the point that better performance is
often observed than with Newton-based procedures \cite{Simoncini.Szyld.Monsalve.13}.
Moreover, ad-hoc parameter selections have been proposed to further enhance
particularly effective approximation spaces \cite{Lin.Simoncini.15}.
This strong numerical evidence however is lacking of any theoretical justification:
the procedure is mainly based on its linear counterpart and therefore
it seems to completely disregard the quadratic term $-XBB^*X$. Nonetheless,
fast convergence to the sought after solution is usually observed.

The aim of this paper is to start an analysis that will lead to a better
understanding of this class of methods. By looking at the computed quantities
from different perspectives, we are able to give new insights into
the role of the approximate solution $X_k$ in the various contexts where
the Riccati equation is extensively studied. We start in section \ref{sec:mor}
with model order
reduction of linear dynamical systems, where approximation by projection is
a recognized important tool, and show that $X_k$ carries information
on the optimal function value in
the reduced control problem.  In section \ref{sec:RKSM} we deepen our knowledge
of $X_k$ and the associated residual, which allows us to derive new expressions
for the residual matrix and justify recently proposed enhancements of a popular space
in model order reduction, that is the rational Krylov
subspace.  A key role in our discussion will be played by the residual matrix,
\begin{eqnarray} \label{eqn:res} 
R_k := A^* X_k + X_k A - X_k BB^* X_k + C^* C  .
\end{eqnarray}
By simple algebra, it is customary to rewrite $R_k$ as
\begin{eqnarray} \label{eqn:resshift}
R_k=(A^*  - X_k BB^*) X_k + X_k (A - BB^* X_k) + C^* C + X_k BB^* X_k ,
\end{eqnarray}
which highlights the occurrence of the matrix $A^*  - X_k BB^*$. This matrix
and its projected version will be ubiquitous in the paper, and are the
true players in the approximation process.
Finally, the connection between the approximation of the matrix equation and the
invariant subspace setting is highlighted in section \ref{sec:invar}.
While our interest was motivated by the good performance of rational Krylov methods,
which are the main focus of section \ref{sec:RKSM},
many of the results in fact hold for more general projection methods.
We believe that our analysis helps provide good ground to characterize projection methods as a 
natural and effective strategy for solving the Riccati equation.

The following notation and definitions will be used. For $X\in \RR^{n\times n}$, $X\ge 0$
means that $X$ is symmetric and positive semidefinite, while $X >0$ means that is symmetric
and positive definite.  A stable matrix is a square matrix with all its eigenvalues in
the open left-half complex plane. 
An $n\times n$
 matrix $A$ is passive if its field of values, $\{z\in\CC\, : \, z=(x^*Ax)/(x^*x), \, 
0\ne x\in\CC^n\}$, is all in the open left-half complex plane. 
$I_n$ denotes the identity matrix of size $n$, and
the subscript will be avoided whenever clear from the context.
A pair $(A, B)$ is controllable if the matrix $[B,AB, \ldots, A^{n-1}B]$ is
full row rank, and $(C, A)$ is observable if $(A^*,C^*)$ is controllable.
A pair $(A,B)$ is stabilizable if there exists a matrix $X$ such that $A-BB^*X$ is stable.
The Euclidean norm $\|\cdot \|$ for vectors and its induced norm for matrices will be used,
together with the Frobenius norm for matrices, defined as $\|A\|_F^2=\sum_{i,j}|a_{i,j}|^2$,
where $A=(a_{i,j})$.

\section{Background on projection methods}
Projection methods usually generate a sequence of nested approximation spaces,
${\cal K}_k\subseteq {\cal K}_{k+1}$, $k \ge 1$, where
an approximate solution is determined. Let the columns of $V_k\in\RR^{n\times d_k}$ span
the space ${\cal K}_k$, where $d_k$ is the space dimension, with
$d_k \le d_{k+1}$. An approximation to $X$ in (\ref{eqn:main}) is
sought as $X_k = V_k Y_k V_k^* \approx X$, where $Y_k$ is determined by imposing
some additional condition. A Galerkin method is characterized by
an orthogonality condition of the residual to the given space, namely 
$R_k \perp {\cal K}_k$, where $R_k$ is as defined in (\ref{eqn:res});
the orthogonality is
with respect to the standard matrix inner product, so that the Galerkin condition reads
\begin{eqnarray}\label{eqn:galerkin}
V_k^* R_k V_k = 0.
\end{eqnarray}
As the subspace grows, the residual is forced to belong to a smaller and smaller
space. When $d_k = n$ then clearly it must be $R_k = 0$ and a solution to 
(\ref{eqn:main}) is determined, in exact arithmetic. 
The main goal is to determine a sufficiently good 
approximate solution $X_k$ for $d_k \ll n$. To obtain $Y_k$ we substitute $X_k$ into
the expression for the residual matrix in (\ref{eqn:galerkin}):
\begin{eqnarray*}
V_k^* ( A^* V_k Y_k V_k^*  + V_k Y_k V_k^* A  - V_k Y_k V_k^* BB^* V_k Y_k V_k^* + C^* C)V_k &=& 0 \\
V_k^*  A^* V_k Y_k   + Y_k V_k^* A V_k  - Y_k V_k^* BB^* V_k Y_k  + V_k^*C^* CV_k &=& 0,
\end{eqnarray*}
where we used that $V_k^*V_k=I_{d_k}$. Setting $T_k = V_k^* A V_k$, $B_k = V_k^*B$ and
$C_k^*=V_k^*C^*$ we see that $Y_k$ can be obtained by solving the reduced Riccati equation
\begin{eqnarray}\label{eqn:Yreduced}
T_k^* Y_k   + Y_k T_k  - Y_k B_kB_k^*  Y_k  + C_k^* C_k = 0 .
\end{eqnarray}
Under the assumption that $A$ is passive, $T_k$ is stable, therefore (\ref{eqn:Yreduced})
admits a unique stabilizing positive semidefinite solution $Y_k$, which is then used
for constructing $X_k$.

The effectiveness of the whole procedure depends on the choice of ${\cal K}_k$.
The approximation spaces explored in the (quite recent)
literature are all based on block Krylov subspaces
generated with $A$ or with rational functions of $A$ and starting term $C^*$
\cite{Jbilou_03},\cite{Heyouni.Jbilou.09},\cite{Simoncini.Szyld.Monsalve.13}.
In section~\ref{sec:RKSM} we will analyze the case of the block rational Krylov
subspace, while the results of the next two sections hold for any approximation
space.

\section{Order reduction of dynamical systems by projection}\label{sec:mor}
The Riccati equation is tightly connected with the time-invariant linear system
\begin{eqnarray}\label{eqn:sys}
\left\{ \begin{array}{l}
\dot x(t) = A x(t) + B u(t), \qquad x(0)=x_0 \\
y(t) = C x(t) , \end{array} \right .
\end{eqnarray}
where $u(t)$ and $x(t)$ are the control (or input) and state vectors, while $y(t)$
is the output vector; $x_0$ is the initial state. 
We note that $x(t)$ also depends
on both $x_0$ and $u(t)$, but this will not be explicitly reported in the notation.
Let us introduce the following quadratic cost 
functional\footnote{Here we consider a simplified version to make an immediate
connection with the Riccati equation stated in (\ref{eqn:main}).}
$$
{\cal J}(u,x_0) = \int_0^\infty (x(t)^*C^*C x(t) + u(t)^* u(t)) dt .
$$
The Riccati equation matrix $X$ is used in the solution of
the following linear-quadratic regulator problem:
$$
\inf_{u} {\cal J}(u,x_0) ,
$$
which consists in finding an {\it optimal control} function $u_*(t)$ associated with the system
(\ref{eqn:sys}),
at which the function ${\cal J}$ attains its infimum. The following well known result connects the
optimal cost problem with the solution of the algebraic Riccati equation (\ref{eqn:main});
see, e.g., the relevant part of \cite[Theorem 16.3.3]{Lancaster.Rodman.95} in our notation.

\vskip 0.1in
\begin{theorem}\label{th:JvsR}
Let the pair $(A,B)$ be stabilizable and $(C,A)$ observable. Then there is a
unique solution $X\ge 0$ of (\ref{eqn:main}). Moreover,

i) For each $x_0$ there is a unique optimal control, and it is given
by $u_*(t) = -B^*X \exp((A - BB^*X)t) x_0$ for $t\ge 0$;

ii) $J(u_*,x_0) = x_0^* X x_0$ for all $x_0 \in \CC^n$.
\end{theorem}

\vskip 0.1in
The optimal control function $u_*(t)$ in the theorem above is in fact determined  as
$u_*(t) = -B^*X x(t)$, giving rise to the
closed-loop dynamical system 
$$
\dot x(t) = (A-BB^*X) x(t), \qquad x(0)=x_0,
$$
whose solution is
$x(t) = \exp((A - BB^*X)t) x_0$ for $t\ge 0$ \cite{Hewer.Kenney.88}. 

A reduced order model aims at representing the given large dynamical system  by means
of a significantly smaller one. This can be done by projecting data onto a significantly
smaller space.  A popular strategy in this class is to use the Rational Krylov 
subspace to reduce the coefficient matrices by projecting them onto an
appropriate vector space \cite{Antoulas.05}. 
The solutions of the reduced system can effectively approximate the original state and control
in case the space trajectories do not occupy the whole state space. In practice,
this means that the original model can be well represented by far fewer
degrees of freedom \cite{Antoulasetal.10}.

A quantity of interest to the control community
that is used to monitor the quality of the reduced system 
is the transfer function, for which a large literature is available; see, e.g., \cite{Antoulas.05},%
\cite{Benner2005a},\cite{Gallivanetal.04},\cite{Schilders2008} and their references. 
Here we focus on the reduction process, and show
that the subspace projection allows one to
determine the optimal control of the reduced dynamical system.
  Let the $d_k\ll n$ orthonormal columns of $V_k \in \RR^{n\times d_k}$ span the
computed subspace, and, as in the previous section,
 let $T_k = V_k^* A V_k$, $B_k =V_k^* B$, $C_k^* = V_k^* C^*$.
Then we can define the reduced order system

\begin{eqnarray}\label{eqn:sys_reduced}
\left\{ \begin{array}{l}
{\dot{\widehat x}}(t) = T_k \widehat x(t) + B_k \widehat u(t), \qquad \widehat x(0)=V_k^* x_0 .\\
\widehat y(t) = C_k \widehat x(t), \end{array} \right .
\end{eqnarray}

Clearly, as $d_k\to n$ the reduced system approaches the original one.
For smaller $d_k$, the quantity $x_k(t) = V_k \widehat x(t)$ is  an approximate
state of the original system. 

\begin{corollary}
The solution matrix $Y_k$ of (\ref{eqn:Yreduced})
 is the unique non-negative solution 
that gives the feedback optimal control $\widehat u_*(t)$, $t\ge 0$,
for the system (\ref{eqn:sys_reduced}).
\end{corollary}

\begin{proof}
Let
$$
\widehat {\cal J}_k(\widehat u, \widehat x_0) = 
\int_0^\infty (\widehat x(t)^* C_k^*C_k  \widehat x(t) + \widehat u(t)^* \widehat u(t)) dt .
$$
be the cost functional associated with (\ref{eqn:sys_reduced}).
 By applying Theorem \ref{th:JvsR}, 
an optimal control for the reduced system is 
$\widehat u_*(t) = -B_k^* Y_k \exp( (T_k - B_kB_k^* Y_k)t) \widehat x_0$,
where $Y_k$ solves the reduced Riccati equation
\begin{eqnarray}\label{eqn:main_reduced}
T_k^* Y + Y T_k - Y B_kB_k^* Y + C_k^* C_k =0 ,
\end{eqnarray}
with the reduced state $\widehat x(t) = \exp( (T_k - B_kB_k^* Y_k)t) \widehat x_0$.
Equation (\ref{eqn:main_reduced}) is precisely the Riccati equation obtained
by Galerkin projection of the original large scale matrix equation (\ref{eqn:main})
 onto the given subspace. 
\end{proof}

Theorem \ref{th:JvsR}(ii) implies
$$
\widehat {\cal J}_k(\widehat u_*, \widehat x_0) = \widehat x_0^* Y_k \widehat x_0 =
x_0^* V_k Y_k V_k^* x_0
=  x_0^* X_k  x_0.
$$
Therefore, if $X_k \to X$ as $d_k \to \infty$, the optimal value of the reduced functional 
yields an estimate to the minimum functional cost via
the approximate solution $X_k=V_k Y_k V_k^*$ to the large Riccati equation.

An approximate control function $u_k(t)$ for the
unreduced functional ${\cal J}(u_k, x_0)$ is obtained directly
using the approximation $X_k$, bypassing the reduced functional $\widehat {\cal J}$.
Indeed, if we assume that the approximate Riccati solution $X_k$ is stabilizing, we can write
\begin{eqnarray}\label{eqn:ux}
u_k(t) = -B^* X_k x_k(t), \qquad {\rm with}\quad x_k(t)= \exp( (A - B B^* X_k)t)  x_0   .
\end{eqnarray}
Substituting $X_k = V_k Y_k V_k^*$ we get
$u_k(t)= -(B^* V_k) Y_k V_k^* \exp( (A - B (B^* V_k) Y_kV_k^*)t)  x_0$.
The question then arises as of whether $u_k$ and $\widehat u_*$ are
related.
Comparing this expression with that of 
$\widehat u_*(t)$, we see that they are close to each other as soon as
$$
\exp( (V_k^*(A  -  B B X_k)V_k t) V_k^* \approx
 V_k^* \exp( (A - B B^* X_k)t) .
$$
Using the expansion of $\exp(z)$ in terms of power series and taking transpose conjugations, 
this approximation can be written as
$$
(A^* - X_k BB^*)^{\ell} V_k \approx V_k \left( V_k^*(A^* - X_k BB^*)V_k\right)^{\ell},
\qquad {\rm for \, any\, }\ell \in \NN.
$$
This approximation becomes an equality as soon as range($V_k$) is an invariant
subspace of $A^* - X_k BB^*$. In general, however, the columns of $V_k$ do not
span an invariant subspace, therefore this connection is not sufficient to
connect the two control functions.
The following proposition does provide a relation between the optimal reduced cost functional value
with the value of the original functional at $u_k$. 

\begin{proposition}
Assume that $A-BB^*X_k$ is stable and that $u_k$ is defined as in (\ref{eqn:ux}).
With the previous notation it holds
$$
|{\cal J}(u_k,x_0) - \widehat {\cal J}_k(\widehat u_*, \widehat x_0) | \le
\frac{\|R_k\|}{2\alpha} x_0^*x_0,
$$
where $\alpha>0$ is such that
$\|e^{(A-BB^*X_k)^*t} \| \le e^{-\alpha t}$ for all $t \ge 0$.
\end{proposition}

\begin{proof}
Using (\ref{eqn:resshift}), let us write the Riccati residual equation as
$$(A-BBX_k^*)^*X_k + X_k(A-BBX_k) + X_kBB^*X_k + C^*C - R_k = 0.$$
Then
\begin{eqnarray*}
{\cal J}(u_k, x_0) &=& \int_0^\infty ( u_k^* u_k + x_k^*C^*Cx_k) dt \\
&=& 
\int_0^\infty x_0^* e^{(A-BB^*X_k)^*t} (X_k BB^*X_k+C^*C) e^{(A-BB^*X_k)t} x_0 dt
\\
&=& x_0^* X_k x_0 + 
\int_0^\infty x_0^* e^{(A-BB^*X_k)^*t} R_k e^{(A-BB^*X_k)t} x_0 dt.
\end{eqnarray*}
From 
$x_0^* X_k x_0 = {\widehat J}_k(\widehat u_*, \widehat x_0)$ and
$|\int_0^\infty x_0^* e^{(A-BB^*X_k)^*t} R_k e^{(A-BB^*X_k)t} x_0 dt| \le
\frac{\|R_k\|}{2\alpha} x_0^*x_0$ the result follows.
\end{proof}

This theorem establishes a linear relation between the matrix equation
residual norm and the distance between the optimal value of the reduced functional
and the value of the approximate unreduced functional. As the residual
norm goes to zero, the two functional values coalesce, and this may occur
for $d_k \ll n$, that is with a projection space of much smaller dimension
than the original one.

We conclude with a remark about the type of approximation space used.
In model order reduction, usually different projection spaces are used from the
 left and from the right, so as to expand both in terms of $C^*$ and $B$.
The connection between this approach and the reduction of the
(symmetric) Riccati equation deserves future analysis.

\section{Control stability properties of the  subspace projection approximation}\label{sec:control}

By using the residual equation,
norm estimates for the error $X-X_k$ can be derived by using classical
perturbations results. In this section we recall these classical estimates,
which can have a different flavor in our setting, where the
{\it perturbations} are not very small in general. Nonetheless, these
results enable us to state that for $d_k$ large enough the approximate solution $X_k$ is
rigorously equipped with all the nice stabilizability properties of
the exact solution.
{Moreover, they can be used to track the progress in the
approximation as the approximation space grows.}

Unlike the linear equation case, a small residual norm
does not necessarily imply a small error, since the Riccati equation
has more than one solution. Therefore, in general an assumption is needed
about the closeness of the approximate solution to the sought after one,
to be able to derive information on the error norm from the residual norm.

Let $X$ be an exact stabilizing solution,
 $E_k = X-X_k$ the error and $R_k = A^* X_k +X_k A -  X_k BB^* X_k + C^* C$ the residual.
Subtract this residual equation from (\ref{eqn:main}). Then
by  adding and subtracting
$X BB^* X_k$ and $E_k BB^* X_k$ in sequence, we obtain
$$
(A^*-X B B^*) E_k + E_k(A - BB^* X) + E_k BB^*E_k + R_k = 0 . 
$$
We observe in passing that the second order term in $E_k$ becomes 
negligeable for $\|E_k\|\ll 1$.
From this Riccati equation for the error,
under certain conditions a bound on the error can be obtained.
To this end we recall the definition of the {\it closed-loop Lyapunov operator}
$$
\Omega_{X}(Z) : = (A-BB^*X)^* Z + Z (A - BB^*X),
$$
and observe that if $H$ is the matrix solving
$(A-BB^*X)^*H+H(A-BB^*X)=-I$, then
$\|H\| = \|\Omega_X^{-1}\| = \max_{Z\ne 0} ( \|\Omega_X^{-1}(Z)\|/\|Z\|)$;
see \cite[Lemma 2]{Kenney.Laub.Wette.90}. Note that $\|\Omega_X^{-1}\|$ is the
reciprocal of the sep operator for the given matrix \cite{Stewart.Sun.90}.
An interesting interpretation of $\|\Omega_X^{-1}\|$
in terms of the {\it damping} of the closed-loop dynamical 
system is also given in \cite{Kenney.Laub.Wette.90}.

\begin{theorem}{\rm \cite{Kenney.Laub.Wette.90}}\label{th:error}
Let $X$ be a symmetric and positive semidefinite solution to (\ref{eqn:main})
such that $A-BB^*X$ is stable. Assume that $\|X-X_k\| < 1/(3\|B\|^2 \|\Omega_X^{-1}\|)$.
If the residual matrix $R_k$ satisfies $4\|B\|^2 \|\Omega_X^{-1}\|^2 \|R_k\| < 1$ then
$$
\|X-X_k \| \le 2 \|\Omega_{X}^{-1}\| \, \|R_k\| .
$$
\end{theorem}
We refer the reader to \cite{Gahinet.Laub.90}  for more refined estimates. 
This bound is a generalization to the nonlinear case
of the well known bound for the (vector) norm
of the error when approximately solving a linear system $Ax= b$.
We note that the ``norm of the inverse'' is replaced here with
the norm of the closed-loop operator inverse, which takes into account
both the linear and the quadratic coefficient matrices.

We next recall a theorem on the sensitivity of the
Lyapunov equation solution. 

\begin{theorem}\label{th:HK}{\rm \cite[Theorem 2.2]{Hewer.Kenney.88}}
Let $A$ be stable and let $H$ satisfy $A^*H+HA=-I$. Let $\Delta A$ satisfy
$\|\Delta A\| < 1/(2\|H\|)$. Then $A+\Delta A$ is stable.
\end{theorem}

This result enables us to state that
if the error $X-X_k$ is small enough, then $X_k$ is stabilizable; a similar result can
also be found in \cite[Lemma 1]{Kenney.Laub.Wette.90}.

\begin{corollary}
Let $A-BB^* X$ be stable and
let $X_k$ be an approximate solution to
(\ref{eqn:main}) and $E_k = X - X_k$. 
If $\|BB^*E_k\|< 1/(2\|\Omega_X^{-1}\|)$, 
then $A-BB^*X_k$ is stable.
\end{corollary}

\begin{proof}
We write $A-BB^*X_k = (A-BB^*X) + BB^*E_k =: \widetilde A + \Delta {\widetilde A}$.
We thus apply Theorem \ref{th:HK} to $\widetilde A$, $\Delta {\widetilde A}$:
$\widetilde A$ is stable by hypothesis; moreover,
if $\|BB^*E_k\| = \|\Delta {\widetilde A}\| < 1/(2\|\Omega_X^{-1}\|)$ then
$\widetilde A + \Delta {\widetilde A}$ is stable.
\end{proof}

Finally, we turn our attention to the special form of the approximate solution,
that is $X_k = V_k Y_k V_k^*$.
The following result shows that after $k$ iterations of a projection method,
the reduced solution matrix $Y_k$ is stabilizing.

\begin{proposition}
Let $T_k$ be stable and $(T_k-B_kB_k^*Y_k, C_k^*)$ controllable. 
Let $Y_k$ be the approximation obtained after $k$ iterations of
the chosen projection method.
Then $T_k - B_kB_k^* Y_k$ is a stable matrix.
\end{proposition}

\begin{proof}
The symmetric matrix $Y_k$ solves the reduced matrix equation
$T_k^* Y +Y T_k - Y B_kB_k^* Y + C_k^* C_k=0$. Rewriting the equation,
$Y_k$ satisfies
$$
(T_k^*-Y_k B_k B_k^*) Y_k + Y_k( T_k - B_kB_k^* Y_k) + Y_k B_k^*B_kY_k+ C_k^* C_k=0,
$$
that is, $Y_k$ formally solves a Lyapunov equation. Since 
$Y_k B_k^*B_kY_k+ C_k^* C_k \ge C_k^* C_k$, Theorem 5.3.2(b) in \cite{Lancaster.Rodman.95}
ensures that the eigenvalues of $T_k-B_kB_k^*Y_k$ all lie in the
open left half-plane, that is the matrix is stable.
\end{proof}

Next result tracks the modification in the approximate solution matrix
$X_k$ as the subspace grows. It is important to realize that in general,
the matrices $Y_{k}$ in the sequence are computed by solving a
new and expanding Riccati equation, therefore the entries of
$Y_k$ and $Y_{k+1}$ are not related by a simple explicit recurrence. 

\begin{proposition}
Let $X_j$ be the approximate solution onto ${\cal K}_j$ for $j=k,k+1$. Then
for $k$ large enough,
$$
\|X_{k+1} - X_k \| \le 2 \|\Omega_{Y_{k+1}}^{-1}\| \|R_k\| .
$$
\end{proposition}

\begin{proof}
We write $X_{k+1} = V_{k+1} Y_{k+1} V_{k+1}^*$ and
$X_{k} = V_{k} Y_{k} V_{k}^*
 = V_{k+1} \check Y_{k+1} V_{k+1}^*$, where $\check Y_{k+1}$ is $Y_k$ padded with extra
rows and columns to match the dimension of $Y_{k+1}$, and we recall that $V_{k+1} = [V_k, \star]$.
Moreover, we set $T_{k+1} = [ T_k, t_{k+1}^{(1)}; (t_{k+1}^{(1)})^* , \star]$.
$Y_{k+1}$ solves the reduced equation
$T_{k+1}^* Y + Y T_{k+1} - Y B_{k+1}B_{k+1}^* Y + C_{k+1}^* C_{k+1} = 0$. Substituting
instead the matrix $\check Y_{k+1}$ we obtain that the residual satisfies
{\footnotesize
\begin{eqnarray*}
\rho_k &:=& T_{k+1}^* \check Y_{k+1} + \check Y_{k+1} T_{k+1} 
- \check Y_{k+1} B_{k+1}B_{k+1}^* \check Y_{k+1} + C_k^* C_k \\
&=& 
V_{k+1}^*( A^* V_{k+1} \check Y_{k+1} V_{k+1}^* +  V_{k+1} \check Y_{k+1} V_{k+1}^* A 
- V_{k+1} \check Y_{k+1} V_{k+1}^* B B^* V_{k+1} \check Y_{k+1} V_{k+1}^*
 + C^* C ) V_{k+1} \\ 
&=& 
V_{k+1}^*( A^* X_k + X_k  A 
- X_k B B^* X_k + C^* C ) V_{k+1} .
\end{eqnarray*}
}%
Therefore, $\|\rho_k\| \le \|R_k\|$. Using Theorem \ref{th:error}, 
if $\|Y_{k+1} - \check Y_{k+1}\| < 1/( 3\|B_{k+1}\|^2 \|\Omega_{Y_{k+1}}^{-1}\|)$
and $\|\rho_{k}\| \le 1/(4\|B_{k+1}\|^2 \|\Omega_{Y_{k+1}}^{-1}\|^2)$ then
$$
\|Y_{k+1} - \check Y_{k+1}\| \le 2 \|\Omega_{Y_{k+1}}^{-1}\| \|\rho_{k}\| .
$$
Noticing that $\|Y_{k+1} - \check Y_{k+1}\| = \|X_{k+1} - X_{k}\|$ the result follows.
\end{proof}

\section{Rational Krylov subspace approximation}\label{sec:RKSM}
%
The approximation quality of projection methods depends on
the choice of the approximation space ${\cal K}_k$. 
In the case of the
Lyapunov and Sylvester equations, a classical choice is
the Krylov subspace
${\cal K}_k = {\rm range}([C^*, A^*C^*, \ldots, (A^*)^{k-1}C^*])$,
first introduced for this problem by Saad in \cite{Saad1990a}. 
Note that in general, $C^*\in\RR^{n\times p}$ satisfies $p\ge 1$, therefore
the space is in fact a ``block'' space, whose dimension is not greater
than $d_k=pk$.
More recently and motivated by the reduction of dynamical systems,
rational Krylov subspaces have shown to be very attractive.
For ${\mathbf s} = [s_1, s_2, \ldots]$,
with $s_j \in \CC^+$, they are given by
$$
{\cal K}_k(A,C^*, {\mathbf s}) := {\rm range}([C^*, (A-s_2 I)^{-1}C^*,  
\ldots,
\prod_{j=1}^{k-1}(A-s_{j+1} I)^{-1}C^*])  .
$$
If the problem data are real, the shifts are included in conjugate pairs. Moreover,
$\Re(s_j)>0$ therefore all inverses exist for $A$ stable.
We remark that the first block of columns generating ${\cal K}_k$ is simply
the matrix $C^*$; this corresponds to using an infinite parameter $s_1 = \infty$ as
first shift, and this will be an assumption throughout. 
Including $C^*$ into the space is crucial for
convergence, since the whole constant matrix term is exactly represented
in the approximation space.  The effectiveness of the space now depends
on the choice of the parameters $s_j$, $j=2, 3, \ldots$. A lot of work has
been devoted to the analysis of ideal shifts, due to the relevance of
rational Krylov subspaces in eigenproblems \cite{Ruhe1984},\cite{Olsson2006},
matrix function evaluations \cite{Guettel.survey.13},\cite{DLZ10},\cite{Guettel.Knizhnerman.11},
and Model Order Reduction \cite{Grimme1997},\cite{Penzl2000b},%
\cite{Druskin.Simoncini.11}; we refer the readers to 
\cite{Simoncini.survey13} and to the references cited above.
We mention that for linear matrix equations, the choice of $s_j \in \{0, \infty\}$
seems to be particularly effective in many cases, since the computational cost of solving
with the coefficient matrix at each iteration can be somewhat mitigated,
without dramatically sacrificing the asymptotic convergence rate. 
Numerical experiments reported in \cite{Simoncini.Szyld.Monsalve.13} show
that for the algebraic Riccati equation this is no longer the case: the general
rational Krylov subspace appears to be superior in all considered examples, in
terms of subspace dimension, if
the shifts are properly selected. This comparison deserves further study \cite{Simoncini.16}.
For the sake of simplicity of exposition or unless it is explicitly
stated, in the rest of this section and its
subsections we assume that $C$ has a single row, that is $p=1$. There is no relevant
difference for $p>1$, except that the same shift is applied to a block of $p$ vectors,
and that the involved matrices have dimensions depending on $pk$.

For $k \ge 1$, the rational Krylov subspace with shifts $s_1, s_2, \ldots, s_{k}$
satisfies the following Arnoldi relation\footnote{The conjugate-transposition
in $T_k^*$ is used for consistency in the notation employed for the reduced Riccati equation.}
(see, e.g., \cite{Deckers.Bultheel.07}, \cite{Lin.Simoncini.13a}):
\begin{eqnarray}\label{eqn:arnoldi}
A^* V_k = V_k T_k^* + \hat v_{k+1} g_k^*, \qquad V_k^*V_k = I,
\end{eqnarray}
where ${\cal K}_k = {\rm range}(V_k)$,  and
$\hat v_{k+1} {\pmb\beta} = v_{k+1} s_{k} - (I-V_kV_k^*) A^* v_{k+1}$ is the QR
decomposition of the right-hand side matrix,
and with $g_k^*= {\pmb\beta} h_{k+1,k} E_k^* H_k^{-1}$. The matrix 
$$
\begin{bmatrix} H_k \\ h_{k+1,k} E_k^* \end{bmatrix}
$$ 
contains the
orthogonalization coefficients that generate the orthonormal columns of $V_{k+1}$ (see, e.g., 
\cite{Druskin.Simoncini.11}). We set $V_1 \beta_0= C^*$, the reduced QR factorization of $C^*$.
 By construction, the matrix $[V_k, \hat v_{k+1}]$ has orthonormal columns as well.

\begin{proposition}\label{prop:Xk_riccati}
The matrix $X_k$ satisfies
the following algebraic Riccati equation
$$
(A^* - \hat v_{k+1} f_k^*) X + X (A - f_k \hat v_{k+1}^*) - X BB^*X + C^* C = 0,
$$
where $f_k = V_k g_k$ and $g_k$ is as in (\ref{eqn:arnoldi}). 
\end{proposition}

\begin{proof}
The residual satisfies 
\begin{eqnarray}
R_k &=& [ V_k, \hat v_{k+1}] 
\begin{bmatrix} 0 & Y_k g_k \\ g_k^* Y_k & 0 \end{bmatrix} 
\begin{bmatrix} V_k^* \\ \hat v_{k+1}^* \end{bmatrix}  \nonumber \\
&=& \hat v_{k+1} g_k^* Y_k V_k^* + V_k Y_k g_k \hat v_{k+1}^* =
\hat v_{k+1} g_k^*V_k^* X_k + X_k V_k g_k \hat v_{k+1}^* . \label{eqn:res_sum}
\end{eqnarray}
Substituting into equation (\ref{eqn:res}) and collecting terms the
result follows.
\end{proof}

Since $\|f_k^* X_k\|  = \|R_k\|/\sqrt{2}$, the modified equation of Proposition~\ref{prop:Xk_riccati}
tends to the original Riccati equation as convergence takes place. 
However, we cannot infer that $X_k$ is close
to $X$ in the backward error sense, since $\hat v_{k+1} f_k^*$ is not small
in general.

\subsection{The adaptive rational Krylov subspace}
Several different selection strategies have been proposed for the shifts $s_j$. In the linear
equation case, Penzl (\cite{Penzl2000b}) suggested a pre-processing for the
computation of a fixed number of shifts, which are then applied cyclically. 
More recently, a greedy adaptive strategy was proposed
in \cite{Druskin.Simoncini.11} for the same class of problems, which determines the 
next shift during the computation, so
that the process can automatically learn from the convergence behavior of the method. The shifts
are selected by minimizing a particular rational function on an approximate and
adaptively adjusted spectral region of $A$.
In \cite{Lin.Simoncini.15} it was observed that for the Riccati equation the inclusion
of information on $BB^*$ during the shift computation 
-- in the form of eigenvalues of $V_k^*(A^*-X_k BB^*)V_k$ -- 
may be beneficial in certain
cases.  In the following
 we aim to justify this choice. To this end, we need to set a rational function framework
that parallels some of the matrix relations obtained in the previous sections.

A relation corresponding to (\ref{eqn:arnoldi}) can be obtained by using orthogonal
rational functions with respect to some inner product; see, e.g.,
 \cite{Deckers.Bultheel.07}. We note that
each $v_j$ can be written as $v_j=\varphi_j(A) c/\|c\|$, for some orthogonal rational function
$\varphi_j = p_j/q_{j-1}$, where $p_j, q_{j-1}$ are polynomials of degree at
most $j$ and $j-1$, respectively. For $j=0$ we define $\varphi_0 =1$.
Let 
$\Phi_{k-1}(\lambda)=[\varphi_0(\lambda), \varphi_1(\lambda), \ldots, \varphi_{k-1}(\lambda)]$.
Then, 
\begin{eqnarray}\label{eqn:Phirat}
\lambda \Phi_{k-1}(\lambda) = \Phi_{k-1}(\lambda) T_k^* + \hat\varphi_k(\lambda) g_k^* ;
\end{eqnarray}
from (\ref{eqn:Phirat}) it follows that $\theta$ is a zero of $\hat\varphi_k$ if
and only $\theta$ is an eigenvalue of $T_k$.
We refer to \cite[section 2.2]{BGV.10} for a similar relation, where 
a different Arnoldi-type relation is used.


A first attempt to justify the use of information from $A - BB^*X_k$ can be
obtained by generalizing the argument in \cite{Druskin.Simoncini.11}, 
working as if the problem were linear.
For the sake of notational simplicity, for the rest of this
section we let ${\cal A}_k = A - BB^*X_k$
and ${\cal T}_k = V_k^* {\cal A}_k V_k = T_k  - B_kB_k^*Y_k$.
Using (\ref{eqn:res}) we can write the residual as
\begin{eqnarray}
R_k &=&(A^*  - X_k BB^*) X_k + X_k (A - BB^* X_k) + C^* C + X_k BB^* X_k \nonumber\\
&=& {\cal A}_k^* X_k + X_k {\cal A}_k + {\cal D}_k {\cal D}_k^*,\label{eqn:resnew}
\end{eqnarray}
where ${\cal D}_k = [C^*, X_k B]$. We observe that all columns of ${\cal D}_k$
belong to $K_k(A^*, C^*,{\mathbf s})$, since ${\cal D}_k = V_k[E_1 \beta_0 , Y_k B_k]$.

\begin{remark}\label{rem:inv}
The rational Krylov subspace $K_k(A^*, C^*, {\mathbf s})$ 
satisfies an Arnoldi-type property for the matrix ${\cal A}_k$.
 Indeed,
\begin{eqnarray*}
{\cal A}_k^* V_k &=& A^* V_k - X_k BB^* V_k \\
&=& V_k T_k^* + \hat v_{k+1} g_k^* - V_k Y_k B_k B_k^* \\
&=& V_k (T_k^*- Y_k B_k B_k^*) + \hat v_{k+1} g_k^*  = V_k {\cal T}_k^* + \hat v_{k+1} g_k^* .
\end{eqnarray*}
\end{remark}
By using the expression of the residual (\ref{eqn:resnew}) as if it were the residual
matrix of a Lyapunov equation, we can follow
the same reasoning as in \cite{Druskin.Simoncini.11} for
the selection of the next shift. However,
as opposed to the linear case, all involved matrices now depend on
the iteration $k$.  To simplify the presentation, in the following argument we assume
that $C^*=c \in\RR^n$.
Consider the shifted system $({\cal A}_k^*-sI) x = c$,
and an approximate solution $x_k \in K_k({\cal A}_k^*, c, {\mathbf s})$.
Then the residual can be written as
\begin{eqnarray}\label{eqn:shifted}
c - ({\cal A}_k^*-sI) V_k ({\cal T}_k^* - s I)^{-1} e_1 \beta_0
=
\frac{\psi_k({\cal A}_k) c}{\psi_k(s)}, \qquad
\psi_k(z) = \prod_{j=1}^k \frac{z - \lambda_j}{z-s_j} ,
\end{eqnarray}
where $\lambda_j$ are the eigenvalues of ${\cal T}_k$.
%
The next shift $s_{k+1}$ is then determined so that
$$
s_{k+1} = \arg \left (\max_{s \in \partial {\mathbb S}_k} \left|\frac 1 {\psi_k(s)}\right| \right ),
$$
where ${\mathbb S}_k \subset \CC^+$ approximates the mirrored spectral region of ${\cal A}_k$,
and $\partial {\mathbb S}_k$ is its border. 
Note that $\psi_k$ is a multiple of $\hat\varphi_k$ in (\ref{eqn:Phirat}).
A major practical difference from
the adaptive procedure in the Lyapunov equation case is that ${\mathbb S}_k$ will change
at each iteration in agreement with the modifications in
 the spectrum of ${\cal A}_k$. 
In fact, thanks to the Arnoldi relation of Remark~\ref{rem:inv}, the unknown spectral region
of ${\cal A}_k$ is replaced with the spectral region of ${\cal T}_k$, which is computable
after the approximate solution $Y_k$ is determined.
This approach is precisely the one explored in
\cite{Lin.Simoncini.15} for the Riccati equation.
 As opposed to an adaptive shift
selection based on $A$ (see, e.g., \cite{Simoncini.Szyld.Monsalve.13}), this
approach includes information on the second order coefficient matrix, which may be
crucial when the term $-BB^*X$ in $A-BB^*X$ significantly modifies the
spectral properties of $A$ (see Example \ref{ex:mod}).
In the next section we give a rigorous formalization of this argument.

\subsection{A new expression for the residual and the choice of shifts}
In \cite{Beckermann.11} a new expression for the residual of the
Sylvester equation was proposed. We extend this expression 
to the case of the Riccati residual matrix. 
The new expression allows an
interpretation of the two-term sum in (\ref{eqn:res_sum})
by means of rational functions.
Note that the result also holds
for $B=0$, therefore  its proof provides a more elementary proof 
for the Lyapunov equation than in \cite{Beckermann.11}.

\begin{proposition} 
\label{lemma:res}
Assume that the columns of $C^*$ belong to Range$(V_k)$, and
let ${\cal T}_k = V_k^* {\cal A}_k V_k = T_k - B_kB_k^*Y_k$.
Then the residual $R_k$ satisfies
$$
R_k = \widehat R_kV_k^* + V_k\widehat R_k^*, \qquad {\rm with} \quad
\widehat R_k = A^* V_k Y_k + V_k Y_k {\cal T}_k + C^* (C V_k) ,
$$
so that $\|R_k\|_F = \sqrt{2}\|\widehat R_k\|_F$.
\end{proposition}

\begin{proof}
By substituting $\widehat R_k$ in the expression for $R_k$ we obtain,
\begin{eqnarray*}
 \widehat R_kV_k^* + V_k\widehat R_k^* & = &
A^*X_k + V_k Y_k T_k V_k^* - V_k Y_k B_k B_k^* Y_k V_k^* + C^*C \\
&& + X_k A
+ V_k T_k^* Y_k V_k^* - V_k Y_k B_k B_k^* Y_k V_k^* + C^*C \\
&=& R_k + 0,
\end{eqnarray*}
where the reduced equation (\ref{eqn:Yreduced}) and
$C^* C V V^* = C^* C$ were used; this proves the first relation. The
norm relation follows from $V_k^* \widehat R_k = 0$, which can
be readily verified.
\end{proof}

\vskip 0.05in

We shall call $\widehat R_k$ the ``semi''-residual matrix.
The proposition above shows that the residual norm of the Galerkin method for
the Riccati equation is the same as that of an associated Sylvester
equation times the constant $\sqrt{2}$.
As a consequence, we can at least formally state that
$V_k Y_kV_k^*$ is a solution to the
Riccati equation (\ref{eqn:main}), that is 
$R_k=0$, if and only if $Z_k=V_k Y_k$ is the solution to the
Sylvester equation
\begin{eqnarray}\label{eqn:sylvAT}
A^* Z + Z {\cal T}_k +  C^* C V_k = 0 ,
\end{eqnarray}
where ${\cal T}_k$ typically has dimensions much smaller than $A$.
Note that this Sylvester equation is in terms of 
$A$ (and not of ${\cal A}_k = A - BB^*X_k$), but also in terms of ${\cal T}_k$.
Let 
\begin{eqnarray}\label{eqn:Phi}
\psi_k(z) = \frac{{\rm det}(zI - T_k)}{\prod_{j=1}^k(z-s_j)}
 = \frac{\prod_{j=1}^k(z-\theta_j)}{\prod_{j=1}^k(z-s_j)} ,
\end{eqnarray}
where $\theta_j$ are the eigenvalues of $T_k = V_k^* A V_k$.
Then the following representation holds for the semi-residual $\widehat R_k$.
The result was first proved for the Sylvester equation
in \cite{Beckermann.11}  and then generalized to the multi-term linear case
in \cite{Beckermann.Kressner.Tobler.13}.
We prove the result for $C^*$ having a single column, the generalization 
to multiple columns can be
obtained by working with each column of $C$, since the whole matrix $C$
is used to build the approximation space.

\begin{theorem}\label{th:semires}
Assume that $p=1$, that is $C^* = c \in \RR^{n}$.
Let $\psi_k$ be the rational function defined in (\ref{eqn:Phi}) and
assume that $ {\cal T}_k = V_k^* {\cal A}_k V_k$ is diagonalizable.
The semi-residual $\widehat R_k$ of Proposition~\ref{lemma:res} satisfies
$$
\widehat R_k = \psi_k(A^*) c c^*V_k (\psi_k(-{\cal T}_k))^{-1} .
$$
\end{theorem}

\begin{proof}
Let ${\cal T}_k = Q \Theta Q^{-1}$, with $\Theta={\rm diag}(\theta_1, \ldots, \theta_{k})$.
 then the result 
follows from standard arguments for shifted linear systems. Indeed,
substituting this decomposition into $\widehat R_k$ in Proposition~\ref{lemma:res}
it follows that
$\widehat R_k Q = A^* V_k Y_k Q + V_k Y_k Q \Theta + c c^*V_k Q$. Let
$Z := V_k Y_k Q = [ z_1, \ldots, z_k]$, $c \eta_j := c c^* V_k Q e_j$
 and $r_j = \widehat R_k Q e_j$, then we have
$$
r_j = (A^* + \theta_j I ) z_j + c \eta_j.
$$
Due to the Galerkin condition, the residuals $r_j$ are all proportional to $\hat v_{k+1}$,
therefore using (\ref{eqn:Phi})
they can be written as $r_j = \hat\varphi_k(A^*) c \eta_j / \hat\varphi_k(-\theta_j)$.
Collecting all columns we get
$\widehat R_k Q = \psi_k(A^*) c c^* V_k Q \psi_k(-\Theta)^{-1}$, where we recall that
$\psi_k$ is a multiple of $\hat\varphi_k$; multiplying from the
right by $Q^{-1}$ the result follows.
\end{proof}

We observe that the expression of the semi-residual generalizes the
residual formula for the shifted system in (\ref{eqn:shifted}) to the case
of matrix equations. The quantity $(\psi_k(-{\cal T}_k^*))^{-1}$
plays the same scaling role as the scalar
$1/\psi_k(s)$ in the shifted system in (\ref{eqn:shifted}).
This new relation thus appears to be of interest on
its own. Indeed, while for linear matrix equations a parallel
with shifted systems had already been
performed (see \cite[section 4.3]{Simoncini.survey13} and references
therein), the residual matrix associated to the special Sylvester
equation (\ref{eqn:sylvAT}) had not been explicitly written down in terms of
polynomials or rational functions.

The new expression for $\widehat R_k$ suggests a way to determine
the next shift $s_{k+1}$.
Indeed, we first recall that the numerator of the rational function $\psi_k$
is the characteristic polynomial of $T_k$, which thus minimizes the numerator
of $\psi_k$ among all monic polynomials of degree $k$. This makes $\|\psi_k(A)c\|$
small among all rational functions $\psi_k$ with fixed denominator and
monic numerator.
With the next shift we 
thus want to make the quantity $(\psi_k(-{\cal T}_k))^{-1}$
smaller in the expression for $\widehat R_k$. Therefore, we
determine for which $z$ in the spectral region of ${\cal T}_k$
 the quantity $(\psi_k(-z))^{-1}$ is large, and 
add a root there for the construction of the next function $\varphi_{k}$.
 Therefore, $s_{k+1}$
is chosen as the solution to the following problem
\begin{eqnarray}\label{eqn:news}
s_{k+1} = {\rm arg}\max_{s\in \partial{\mathbb S}_k} 
\left | \frac {1}{\psi_k(s)}\right | =
{\rm arg}\max_{s\in \partial{\mathbb S}_k}
\left| \frac {\prod_{j=1}^k(s-s_j)}{{\rm det}(sI - {\cal T}_k)} \right |,
\end{eqnarray}
where here ${\mathbb S}_k$ is a region enclosing the eigenvalues of $-{\cal T}_k$
and $\partial{\mathbb S}_k$ is its border. This approach should be
compared with the original algorithm that uses $T_k$ instead. 
 This modified selection strategy can be implemented
very easily, with a slight modification of the original algorithm in
\cite{Druskin.Simoncini.11}: the algorithm needs to compute the eigenvalues of
$ {\cal T}_k = T_k - B_k B_k^* Y_k$ instead of those of $T_k$ to determine
the corresponding convex hull. It is interesting to observe that for $A$
Hermitian, working with the non-Hermitian matrix ${\cal T}_k$ appears to be
more complex than working with the Hermitian matrix $T_k$.
On the other hand, the matrix ${\cal T}_k$ has a key role in the Riccati semi-residual 
matrix, and it takes into account the nonlinear term in the original equation.
Clearly, if the convex hulls of $T_k$ and ${\cal T}_k$ are similar,
and the same for those of $A$ and of $A-BB^*X$, then no major differences
will be observed between the two selection strategies. In other words, if the
field of values are similar, then the projection method based only on the linear
part will be able to decrease $\|\widehat R_k\|$ with a similar convergence rate.

We next report an example illustrating the
expected behavior of the rational Krylov method with or without the
inclusion of the term $- B_k B_k^* Y_k$ in the computation of the spectral
region in (\ref{eqn:news})\footnote{The Matlab (\cite{matlab7}) code for 
the rational Krylov subspace method for the Riccati equation is available at the author's webpage
{\tt http://www.dm.unibo.it/{\~{ }}simoncin/software.html}.}.

\begin{example}\label{ex:mod}
We consider a small built-up example, where $A$ is the Toeplitz matrix
$A=-{\tt toeplitz}(-1,-1.5,\underline{2.8},1,1,1)$ of size
n=700 (this small size allows us to easily compute all quantities for this
theoretical analysis). Moreover, $B=t {\mathbf 1}$ and $C=[1,-2,1,-2,1,-2,...]$;
this example is motivated by an example with similar data in \cite{Lin.Simoncini.15}.
The parameter $t$ takes the values $t_j = 5\cdot 10^{-j}$, so that for $j=3$,
$\|B\| \approx 1$.
The left plot of Figure~\ref{fig:rksm_1} shows the convergence history
(relative residual norm) of the rational Krylov method for each of the three different
values of $t$, when the shifts are adaptively computed on the spectral region of ${\cal T}_k$,
as in (\ref{eqn:news}).
The right plot of Figure~\ref{fig:rksm_1} shows the modification of the
convex hull of $A^* - X BB^*$ as $t$ varies. For the larger values of $t_j$,
the magnitude of $B$ significantly influences the spectral convex hull; by using the modified
shift computation strategy, the method is able to adapt to this change and capture
the new problem features. We remark that by using spectral information of $T_k$ instead,
the method takes about 12 iterations to converge, irrespective of the value of $t$.
We notice that for $B$ of rank one, the matrix $XBB^*$ is also rank one, with
a real positive eigenvalue whose magnitude depends on $B$ and thus on $t$. 
For $\|B\|$ large, Figure~\ref{fig:rksm_1} shows that for this example
only one eigenvalue of $A^*$ is significantly perturbed in $A^*-XBB^*$, causing
the extension of the original spectrum to the left, by an amount depending on $t$.
\end{example}

\begin{figure}[htb]
\centering
\includegraphics[width=2.0in, height=2.0in]{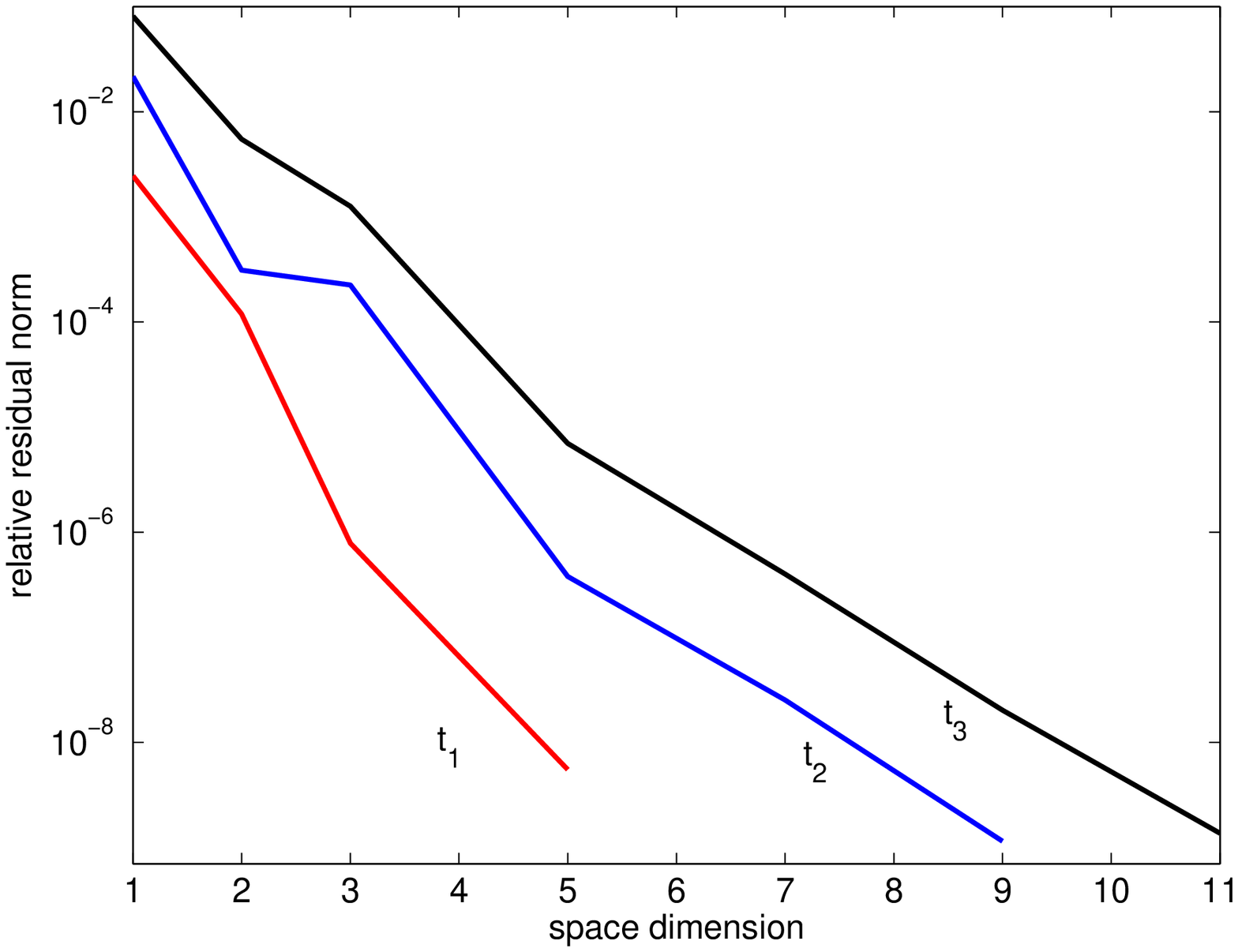}
\includegraphics[width=2.0in, height=2.0in]{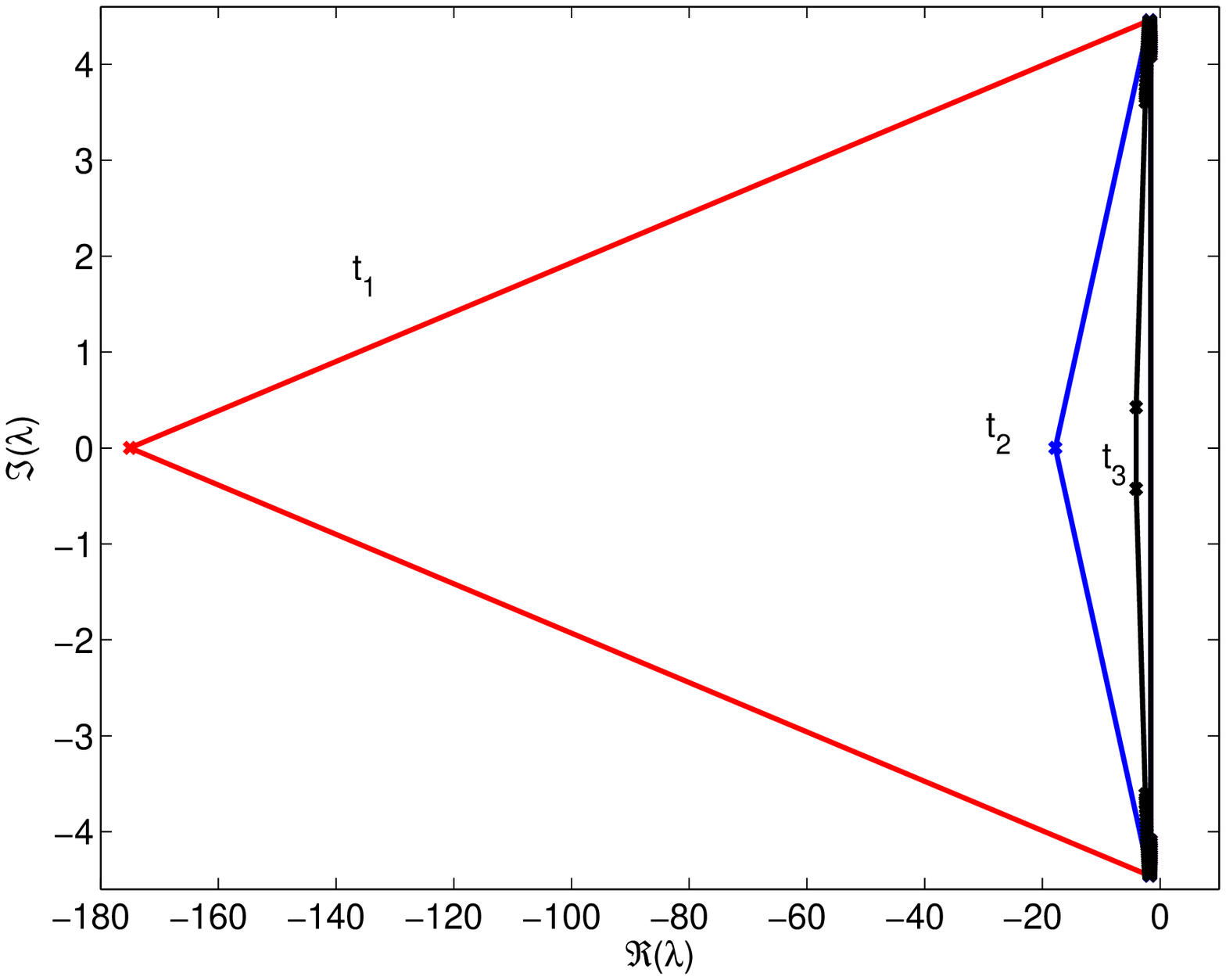}
\caption{{Example \ref{ex:mod}. Left: Convergence history of rational Krylov method
with modified shift selection as $t$ varies. Right: convex hull of $A^*-X BB^*$
as $t$ varies.  } \label{fig:rksm_1}
}
\end{figure}

Example~\ref{ex:mod} shows that for these data, the magnitude of $B$ influences the
residual convergence of the modified method in a counterintuitive way: the larger
its norm, the faster the method convergence. 
By using the modified shift selection, the isolated eigenvalue of $A^* - XBB^*$
(see Figure~\ref{fig:rksm_1}) is readily located, and the residual is forced
to be small in that region as well.

In the next example we explore the influence of the nonsymmetry of ${\cal T}_k$
in the shift computation, when $A$ is symmetric.

\begin{example}\label{ex:sym}
We consider the same setting as for Example~\ref{ex:mod}, except that now 
$A = A_0 \otimes I_{n_0} + I_{n_0} \otimes A_0$,
with $A_0 = {\tt toeplitz}(1, \underline{-2}, 1) \in \RR^{n_0\times n_0}$, with $n_0=30$
and $\otimes$ the Kronecker product,
giving rise to a $900\times 900$ symmetric negative definite matrix. These data
 represent the scaled finite difference
discretization of the Laplacian on the unit square with homogeneous boundary
conditions. As $t$ varies, we compare the performance of the method when ${\mathbb S} \subset\RR$ is
associated with the symmetric matrix $T_k$, with the case when ${\mathbb S} \subset\CC$
due to the use of ${\cal T}_k$; to emphasize this dependence will shall use
${\mathbb S}({\rm T}_k)$ and ${\mathbb S}({\cal T}_k)$, respectively. 
Table~\ref{tab:sym} shows the space dimension required by the two approaches
to reach an absolute residual norm of $10^{-9}$. Shown are also the absolute residual
and error norms at convergence, and the norm of the exact solution.
We report that all computed shifts were real also for ${\cal T}_k$.
The table shows that
the number of iterations for the residual to converge is always smaller when
${\mathbb S}({\cal T}_k)$ is used, and it decreases with  the
magnitude growth of the $B$ term, as in the previous example. 
We also notice that when using $T_k$, the final error is significantly
smaller than in the modified version of the method; apparently, the
residual lags behind in convergence, when ${\mathbb S}(T_k)$ is used.

Figure~\ref{fig:sym} displays the residual convergence history for the two approaches,
as $t$ varies. The initial steep phase of the residual in the modified approach
is granted by the fact that the approximation space immediately locates
the isolated eigenvalue, and that the residual appears to have a large
component in the corresponding eigendirection. 
After that, the convergence behavior depends on the rest of the spectrum.
The original solver maintains the same convergence rate for all values of $t$.

\begin{table}
{
\centering
\begin{tabular}{|llrccc|} \hline
 $t$ & Spectral  & Space  &   $\|R_k\|_F$  &  $\| X-X_k\|_F$ &$\|X\|_F$ \\
   &  Region  & dim.  &    &   &    \\ 
\hline
$10^3$ &  ${\mathbb S}({\rm T}_k)$ & 21  & 1.8500e-10 & 1.6646e-13 & 4.9999e-03\\

     &  ${\mathbb S}({\cal T}_k)$ & 3 &  8.5599e-10 & 1.4389e-10 & \\
\hline
 $10^2$ &  ${\mathbb S}({\rm T}_k)$ & 23 & 3.1915e-10 & 3.0155e-13 & 4.9994e-02 \\

     &  ${\mathbb S}({\cal T}_k)$ & 7 &  4.9612e-10 &  1.0148e-10 & \\
\hline
  10 &  ${\mathbb S}({\rm T}_k)$ & 25 &  9.6706e-10 & 2.5302e-13 & 4.9938e-01 \\

     &  ${\mathbb S}({\cal T}_k)$ & 9 & 9.0853e-10 & 2.2998e-10 & \\
\hline
\end{tabular}
\caption{Example \ref{ex:sym}. Comparison of performance for $A$ symmetric.
Number of iterations for the two variants for the relative residual norm 
and final accuracies to go below $10^{-9}$.
\label{tab:sym}}
}
\end{table}
\end{example}

\begin{figure}[htb]
\centering
\includegraphics[width=2.0in, height=2.0in]{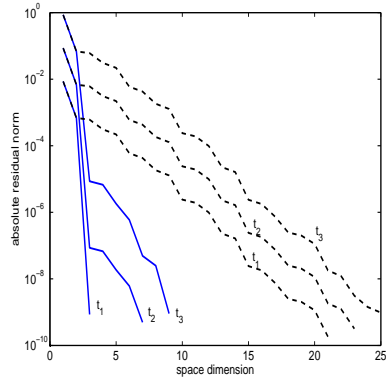}
\caption{{Example \ref{ex:sym}. Convergence history of rational Krylov method
with and without modified shift selection as $t$ varies. 
Solid curves: use of ${\cal T}_k$.
Dashed curves: use of ${T}_k$.
  } \label{fig:sym}
}
\end{figure}

%

By generalizing field of values results in 
\cite{Druskin.Knizhnerman.Simoncini.11},\cite{Beckermann.11}
it may be possible to exploit the semi-residual form to analyze the convergence
of the method, and its dependence on ${\cal T}_k$. A shortcoming in the analysis
is that the field of values of the non-Hermitian matrix
${\cal T}_k$ depends on $k$, and that its relation with
the field of values of $A^* - X BB^*$ is not easy to formalize,
especially at an early stage of the convergence history. 
Resorting to the residual expression in (\ref{eqn:resnew}), it
is possible to exploit some of the results available in the literature for
the Lyapunov equation. For instance,
if the field of values of $A^* - X BB^*$ and of $A^* - X_k BB^*$
is contained in a disk of center $c>0$ and radius equal to one for all $k$, then
using \cite[Theorem 4.11]{Druskin.Knizhnerman.Simoncini.11} we can state that
the error satisfies
$$
\overline{\lim}_{k\to\infty} \|X - X_k\|^{\frac 1 k} \le
\frac{2 c^2 + c-1 - (2c+1)\sqrt{c^2-1}}{c+1+\sqrt{c^2-1}} = : \gamma.
$$
The following example shows that this asymptotic bound can be descriptive
of the actual behavior.

\begin{example}\label{ex:grcar}
We consider $A=-1/(3.2)A_0 - I$ where
$A_0$ is the Grcar matrix, $A_0={\tt toeplitz}(-1,\underline{1},1,1,1)\in\RR^{n\times n}$,
$n=1600$, 
$C={\mathbf 1}/\|{\mathbf 1}\|$ and $B\in\RR^{n\times p}$, $p=20$ with normally
distributed random numbers, normalized so that its norm is about $5\cdot 10^{-2}$.
The left plot of Figure~\ref{fig:grcar} shows the computed spectrum of $A$ 
(`$\times$' symbol), that of $A^*-XBB^*$ (`$\circ$' symbol), the border of
the field of values of both $A$ and $A^*-XBB^*$ (thin line), and 
the circle of center $c=1.25$ and radius one,
enclosing the field of values. The right plot of Figure~\ref{fig:grcar} displays
the error norm history of the modified method (dashed line), and $10^{-2} \gamma^k$,
The convergence rate is well captured by the theoretical estimate $\gamma$
at the early stage of the iterations.

\end{example}

\begin{figure}[htb]
\centering
\includegraphics[width=2.0in, height=2.0in]{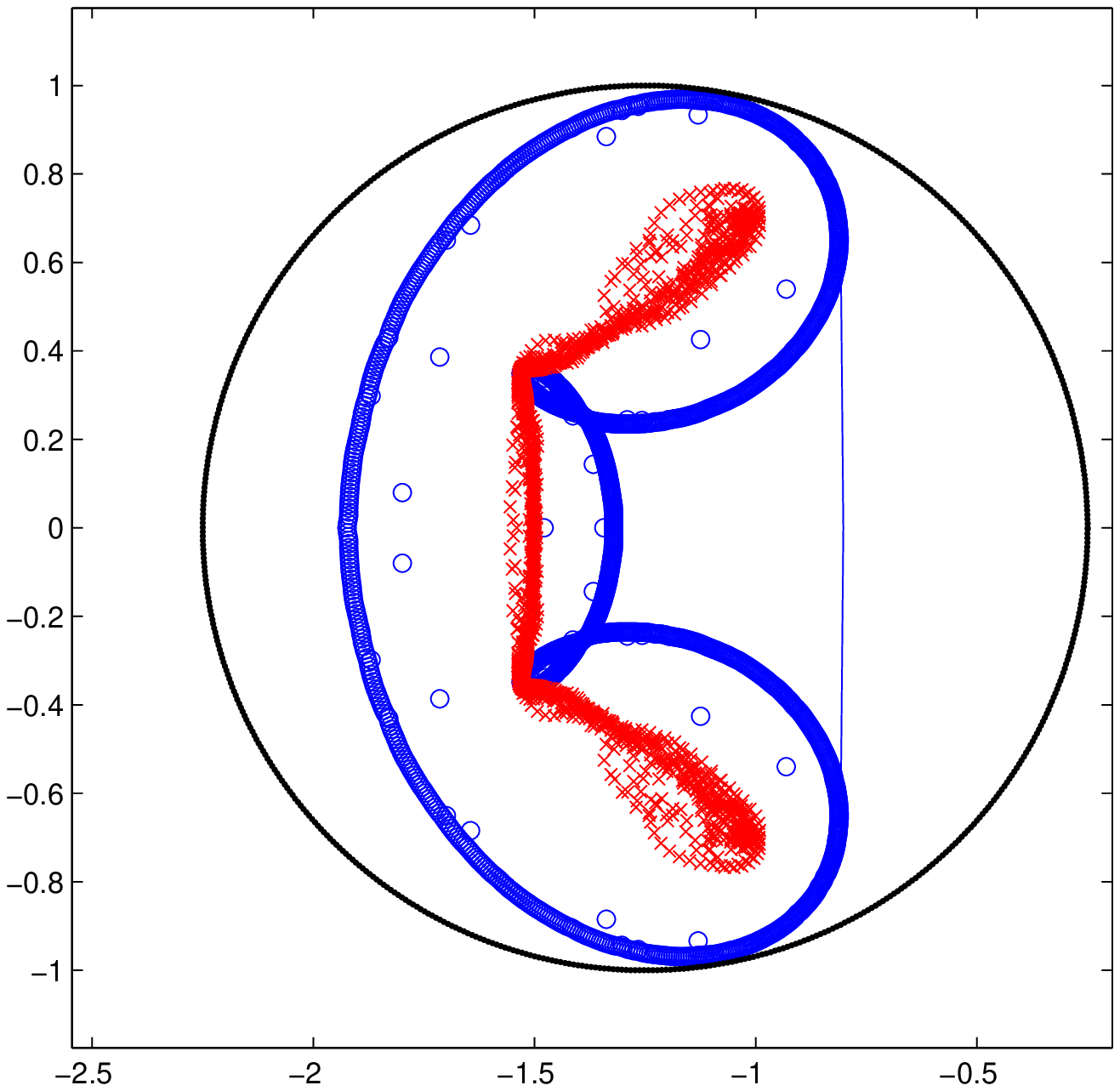}
\includegraphics[width=2.0in, height=2.0in]{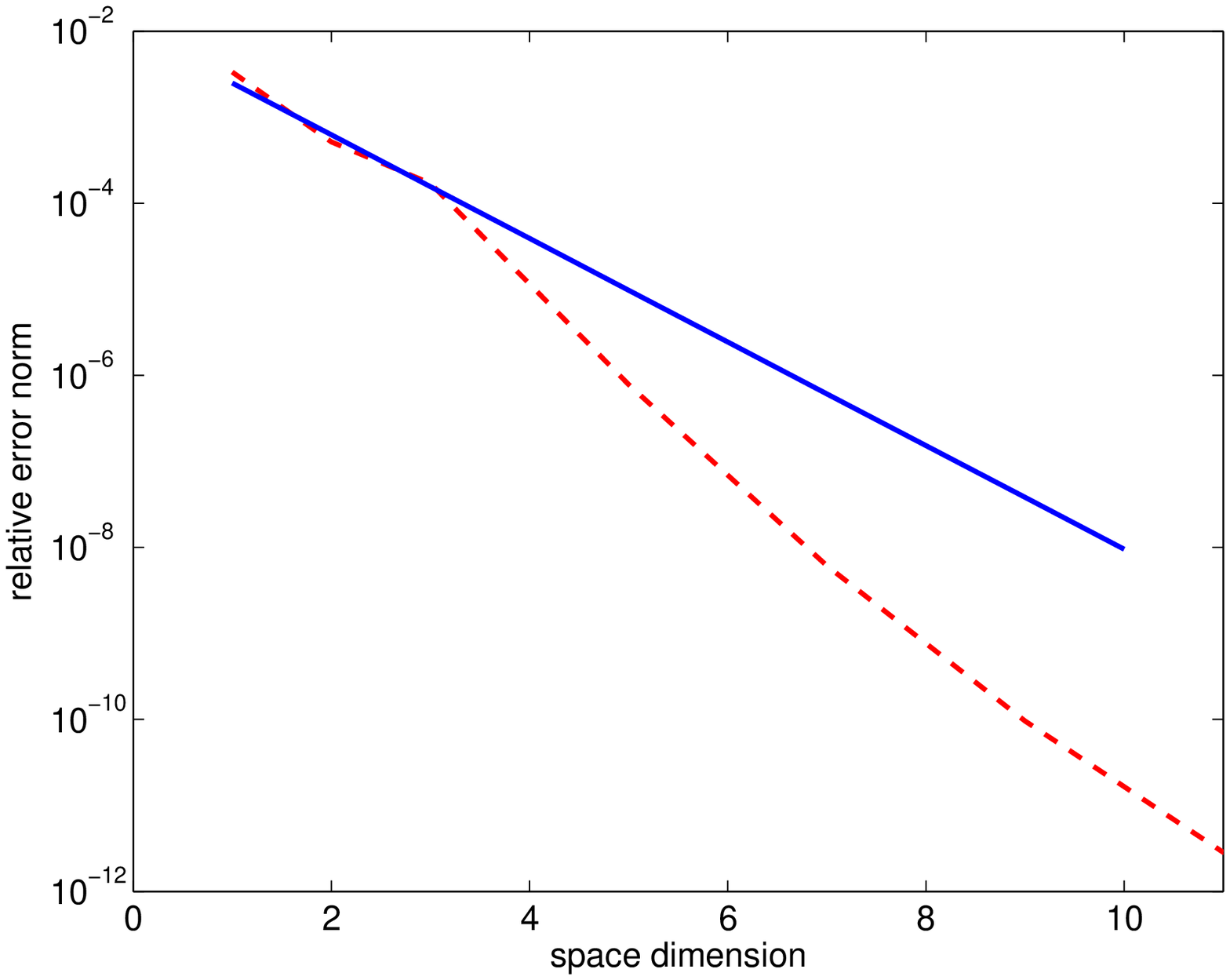}
\caption{Example \ref{ex:grcar}. Left: Field of values and eigenvalue location.
Right: Error norm convergence history and estimate $\gamma^k$.
   \label{fig:grcar} }
\end{figure}

%
%

\section{Approximation of an invariant subspace}\label{sec:invar}
In this section we discuss the natural, albeit gone
almost unnoticed, role of the approximation matrix  $X_k$ in the eigenvalue
context.
The problem of solving the large scale algebraic Riccati equation for $X\ge 0$
can be transformed into the problem of computing an approximate basis for
the stable invariant subspace of the following Hamiltonian matrix (see, e.g., 
\cite{Lancaster.Rodman.95})
\begin{eqnarray}\label{eqn:Hmain}
{\cal H} = 
\begin{bmatrix}
A & -BB^* \\ -C^* C & -A^*
\end{bmatrix} .
\end{eqnarray}
Several different approaches have been devised to this end, see, e.g.,
\cite{Amodei.Buchot.10},\cite{Benner.Bujanovic.16},\cite{Lin.Simoncini.15} and
references therein.
Here we show that the projection process described in the previous
sections can be equivalently applied to this context, providing
further motivation for the method.

Let $X_k$ be the approximate solution to (\ref{eqn:main}) obtained by the
rational Krylov subspace method.
For some $L \in \RR^{n\times n}$ consider the eigenvalue residual
$$
{\cal S}_k =
\begin{bmatrix}
A & -BB^* \\ -C^* C & -A^*
\end{bmatrix}
\begin{bmatrix}
I  \\ X_k
\end{bmatrix} -
\begin{bmatrix}
I  \\ X_k
\end{bmatrix} L.
$$
For $L= A- B^*B X_k$, the invariant space residual ${\cal S}_k$ and the
matrix equation residual $R_k$ can be easily related, since
$$
{\cal S}_k =
\begin{bmatrix}
A & -BB^* \\ -C^* C & -A^*
\end{bmatrix} 
\begin{bmatrix}
I  \\ X_k 
\end{bmatrix} -
\begin{bmatrix}
I  \\ X_k 
\end{bmatrix} ( A- B^*B X_k) 
=
\begin{bmatrix}
0  \\ R_k 
\end{bmatrix} , 
$$
so that
$\|{\cal S}_k\| = \|R_k\|$. As a consequence of Proposition \ref{prop:Xk_riccati}
the following result holds.

\begin{proposition}
The columns of the matrix $[I;X_k]$ span an invariant subspace of the matrix
$$
{\cal H}_k = \begin{bmatrix}
A- f_k \hat v_{k+1}^* & -BB^* \\ -C^* C & -(A - f_k\hat v_{k+1}^*)^*
\end{bmatrix} ,
$$
and the spectrum of $T_k^* - Y_k B_k B_k^*$ is a subset of the spectrum
of $A^* - X_k BB^* - \hat v_{k+1} f_k^*$.
\end{proposition}

\begin{proof}
Writing the eigenresidual
$$
{\cal S}_k =
\begin{bmatrix}
A- f_k \hat v_{k+1}^* & -BB^* \\ -C^* C & -(A - f_k\hat v_{k+1}^*)^*
\end{bmatrix} 
\begin{bmatrix}
I  \\ X_k 
\end{bmatrix} -
\begin{bmatrix}
I  \\ X_k 
\end{bmatrix} ( A- f_k \hat v_{k+1}^*- B^*B X_k) 
$$
and using Proposition \ref{prop:Xk_riccati} we readily see that ${\cal S}_k = 0$.

To prove the second assertion, we use the Arnoldi relation in (\ref{eqn:arnoldi}).
Let $(\theta,z)$ be an eigenpair of $T^* - Y_k B_k B_k^*$. Then
\begin{eqnarray*}
(A^* - \hat v_{k+1} f_k^* - X_k BB^* ) V_k z &=&
(A^* V_k - \hat v_{k+1} f_k^*V_k - X_k BB^*V_k )  z  \\
&=& (V_k T_k^* + \hat v_{k+1} g_k^* - \hat v_{k+1} f_k^*V_k - X_k BB^*V_k )  z  \\
&=& (V_k T_k^* - X_k BB^*V_k )  z = V_k (T_k^* - Y_k B_k B_k^*) z = V_k z \theta,
\end{eqnarray*}
and the result follows.
\end{proof}

The result above states that the approximate Riccati solution is associated
with an invariant subspace of a modification of the original matrix in (\ref{eqn:Hmain}), 
and that the spectrum of $T_k^* - Y_k B_k B_k^*$ is a portion of
the spectrum of this modified problem.
These properties are a consequence of the Arnoldi relation (\ref{eqn:arnoldi}), which indeed
states that $V_k$ is an invariant subspace basis of a modification of $A^*$, namely
of $A^* - \hat v_{k+1} f_k^*$. What is noticeable in our context is that
we can relate the spectral region over which we seek the next shift in (\ref{eqn:news})
with the spectral region of a relevant matrix back in $\RR^n$.

The  approximation process leading to the computation of $Y_k$
can be interpreted as a Galerkin method for the eigenvalue
problem associated with ${\cal H}$. Consider the space
$$
{\mathbb V}_k = {\rm range}\left ( \begin{bmatrix} V_k & 0 \\ 0 & V_k\end{bmatrix} \right ) =:
{\rm range}({\cal V}_k).
$$
Then by projecting ${\cal H}$ onto the space we obtain, 
$$
{\cal V}_k^* {\cal H} {\cal V}_k =
\begin{bmatrix}
V_k^*AV_k & - V_k^* BB^*V_k \\ -V^* C^*CV_k & - V_k^*A^*V_k
\end{bmatrix}
=
\begin{bmatrix}
T_k & - B_kB_k^* \\ -C_k^*C_k & - T_k^*
\end{bmatrix}.
$$
The block matrix on the right-hand side is the Hamiltonian matrix
associated with the reduced system. 
Using the reduced Riccati equation, it holds that
$$
\begin{bmatrix}
T_k & - B_kB_k^* \\ -C_k^*C_k & - T_k^*
\end{bmatrix}
\begin{bmatrix}
I \\ Y_k 
\end{bmatrix} 
=
\begin{bmatrix}
I \\ Y_k 
\end{bmatrix} 
(T_k - B_kB_k^* Y_k) ,
$$
with $Y_k$ stabilizing.
%
In terms of original space dimensions,
let ${\cal V}_k [I;Y_k] = [V_k; V_k Y_k]$ be the computed approximate eigenbasis.
Then the residual is given by
\begin{eqnarray*}
{\cal S}_k &=& 
\begin{bmatrix}
A & -BB^* \\ -C^* C & -A^*
\end{bmatrix} 
\begin{bmatrix}
V_k  \\ V_k Y_k 
\end{bmatrix} -
\begin{bmatrix}
V_k  \\ V_k  Y_k
\end{bmatrix} ( T_k- B_k B_k^* Y_k)  .
\end{eqnarray*}
It readily follows that the eigenresidual is orthogonal to the generated space, that is
it holds that $({\cal V}_k)^* {\cal S}_k = 0$, therefore it satisfies a 
standard Galerkin condition.
As a consequence, for $T_k - B_k B_k^* Y_k$ stable,  ${\cal V}_k [I;Y_k]$
approximates a basis of a stable invariant subspace of the matrix ${\cal H}$ in the sense of
Galerkin projection methods.

\section{Conclusions}\label{sec:conclusions}
By looking at the problem from different but highly related perspectives,
we have shown that projection methods are a natural device
for solving the algebraic Riccati equation.
In particular, the reduced equation solves a reduced linear-quadratic
optimization problem, as is typical of model order reduction techniques.
By using classical arguments, we have related the residual with the error of
the current approximation. Moreover, we have derived a new expression for the
residual in terms of rational functions; this expression allows us to justify
recent algorithmic strategies for the choice of the shift parameters
used in the construction of the approximation space.
In addition, this expression highlights the role of the quadratic term,
and explains why it often happens that good convergence occurs even
without taking the quadratic term into account during the construction
of the approximation space.
The new relations for the residual in terms of rational functions can
be the starting point for a convergence analysis of the method. 
We notice that while we have focussed on generic rational Krylov subspaces
in section \ref{sec:RKSM}, many of the stated results hold for other choices of
approximation spaces, and in particular for polynomial and extended 
Krylov subspaces.

Finally, we have shown that the computed quantities correspond to
a Galerkin approximation of the eigenvalue problem associated with
the Hamiltonian matrix of the dynamical system.

\section*{Acknowledgements}
We would like to thank Dario Bini and Daniel Szyld for insightful comments.
This research is supported in part by the FARB12SIMO grant of the Universit\`a di Bologna,
and by INdAM-GNCS under the
2016 Project  \emph{Equazioni e funzioni di matrici con struttura: analisi
e algoritmi}.

\bibliography{%
/home/valeria/Bibl/Biblioteca}

\begin{thebibliography}{10}

\bibitem{Amodei.Buchot.10}
{\sc L.~Amodei and J.-M. Buchot}, {\em An invariant subspace method for
  large-scale algebraic {R}iccati equation}, Applied Numerical Mathematics, 60
  (2010), pp.~1067--1082.

\bibitem{Antoulasetal.10}
{\sc A.~Antoulas, C.~Beattie, and S.~Gugercin}, {\em {Interpolatory model
  reduction of large-scale dynamical systems}}, in {Efficient Modeling and
  Control of Large-Scale Systems}, J.~Mohammadpour and K.~Grigoriadis, eds.,
  Springer-Verlag, February 2010.
\newblock ISBN 978-1-4419-5756-6.

\bibitem{Antoulas.05}
{\sc A.~C. Antoulas}, {\em Approximation of large-scale {Dynamical Systems}},
  {Advances in Design and Control}, SIAM, Philadelphia, 2005.

\bibitem{Beckermann.11}
{\sc B.~Beckermann}, {\em {An Error Analysis for Rational Galerkin Projection
  applied to the Sylvester Equation}}, SIAM J. Numer. Anal., 49 (2011),
  pp.~2430--2450.

\bibitem{BGV.10}
{\sc B.~Beckermann, S.~G{\"u}ttel, and R.~Vandebril}, {\em On the convergence
  of rational ritz values}, SIAM J. Matrix Anal. Appl., 31 (2010),
  pp.~1740--1774.

\bibitem{Beckermann.Kressner.Tobler.13}
{\sc B.~Beckermann, D.~Kressner, and C.~Tobler}, {\em An error analysis of
  {G}alerkin projection methods for linear systems with tensor product
  structure}, SIAM J. Numer. Anal., 51 (2013), pp.~3307--3326.

\bibitem{Benner.Bujanovic.16}
{\sc P.~Benner and Z.~Bujanovic}, {\em On the solution of large-scale algebraic
  {Riccati} equations by using low-dimensional invariant subspaces}, Linear
  Algebra and Its Applications, 488 (2016), pp.~430--459.

\bibitem{Benner.Li.Penzl.08}
{\sc P.~Benner, J.-R. Li, and T.~Penzl}, {\em Numerical solution of large-scale
  {L}yapunov equations, {R}iccati equations, and linear-quadratic optimal
  control problems}, Num. Lin. Alg. with Appl., 15 (2008), pp.~1--23.

\bibitem{Benner2005a}
{\sc P.~Benner, V.~Mehrmann, and D.~S. (eds)}, {\em {Dimension Reduction of
  Large-Scale Systems}}, Lecture Notes in Computational Science and
  Engineering, Springer-Verlag, Berlin/Heidelberg, 2005.

\bibitem{Benner.Saak.10}
{\sc P.~Benner and J.~Saak}, {\em A {Galerkin-Newton-ADI} method for solving
  large-scale algebraic {R}iccati equations}, Tech. Rep. SPP1253-090, Deutsche
  Forschungsgemeinschaft - Priority Program 1253, 2010.

\bibitem{Binietal.book.12}
{\sc D.~Bini, B.~Iannazzo, and B.~Meini}, {\em Numerical Solution of Algebraic
  Riccati Equations}, SIAM, Philadelphia, 2012.

\bibitem{Deckers.Bultheel.07}
{\sc K.~Deckers and A.~Bultheel}, {\em Rational {K}rylov sequences and
  orthogonal rational functions}, tech. rep., {Department of Computer Science,
  K.U.Leuven}, 2007.

\bibitem{Druskin.Knizhnerman.Simoncini.11}
{\sc V.~Druskin, L.~Knizhnerman, and V.~Simoncini}, {\em Analysis of the
  rational {K}rylov subspace and {ADI} methods for solving the {L}yapunov
  equation}, SIAM J. Numer. Anal., 49 (2011), pp.~1875--1898.

\bibitem{DLZ10}
{\sc V.~Druskin, C.~Lieberman, and M.~Zaslavsky}, {\em On adaptive choice of
  shifts in rational {K}rylov subspace reduction of evolutionary problems},
  SIAM J. Sci. Comput., 32 (2010), pp.~2485--2496.

\bibitem{Druskin.Simoncini.11}
{\sc V.~Druskin and V.~Simoncini}, {\em Adaptive rational {Krylov} subspaces
  for large-scale dynamical systems}, Systems and Control Letters, 60 (2011),
  pp.~546--560.

\bibitem{FHS.09}
{\sc F.~Feitzinger, T.~Hylla, and E.~W. Sachs}, {\em Inexact
  {K}leinman-{N}ewton method for {R}iccati equations}, SIAM J. Matrix Anal.
  Appl., 31 (2009), pp.~272--288.

\bibitem{Gahinet.Laub.90}
{\sc P.~Gahinet and A.~J. Laub}, {\em {Computable bounds for the sensitivity of
  the algebraic Riccati equation}}, {SIAM J. Control and Opt.}, 28 (1990),
  pp.~1461--1480.

\bibitem{Gallivanetal.04}
{\sc K.~Gallivan, A.~Vandendorpe, and P.~V. Dooren}, {\em {Sylvester equations
  and projection-based model reduction}}, J. Comput. Appl. Math., 162 (2004),
  pp.~213--229.

\bibitem{Grasedyck.08}
{\sc L.~Grasedyck}, {\em Nonlinear multigrid for the solution of large-scale
  {R}iccati equations in low-rank and {$\mathcal H$}-matrix format}, Numer.
  Linear Algebra Appl., 15 (2008), pp.~779--807.

\bibitem{Grasedyck.Hackbusch.Khoromskij.03}
{\sc L.~Grasedyck, W.~Hackbusch, and B.~Khoromskij}, {\em Solution of large
  scale algebraic matrix {R}iccati equations by use of hierarchical matrices},
  Computing, 70 (2003), pp.~121--165.

\bibitem{Grimme1997}
{\sc E.~Grimme}, {\em Krylov projection methods for model reduction}, PhD
  thesis, The University of Illinois at Urbana-Champaign, 1997.

\bibitem{Guettel.survey.13}
{\sc S.~G{\"u}ttel}, {\em Rational {K}rylov approximation of matrix functions:
  Numerical methods and optimal pole selection}, GAMM Mitteilungen, 36 (2013),
  pp.~8--31.

\bibitem{Guettel.Knizhnerman.11}
{\sc S.~G{\"u}ttel and L.~Knizhnerman}, {\em Automated parameter selection for
  rational {A}rnoldi approximation of {M}arkov functions}, Proc. Appl. Math.
  Mech., 11 (2011), pp.~15--18.

\bibitem{Hewer.Kenney.88}
{\sc G.~Hewer and C.~Kenney}, {\em The sensitivity of the stable {L}yapunov
  equation}, SIAM J. Control and Optimization, 26 (1988), pp.~321--344.

\bibitem{Heyouni.Jbilou.09}
{\sc M.~Heyouni and K.~Jbilou}, {\em An extended {Block Krylov method} for
  large-scale continuous-time {algebraic Riccati equations}}, ETNA, 33
  (2008-2009), pp.~53--62.

\bibitem{Jbilou_03}
{\sc K.~Jbilou}, {\em {Block Krylov subspace methods for large algebraic
  Riccati equations}}, {Numerical Algorithms}, 34 (2003), pp.~339--353.

\bibitem{Kenney.Laub.Wette.90}
{\sc C.~Kenney, A.~J. Laub, and M.~Wette}, {\em Error bounds for {N}ewton
  refinement of solutions to algebraic {R}iccati equations}, Math. Control
  Signals Systems, 3 (1990), pp.~211--224.

\bibitem{Kleinman_68}
{\sc D.~L. Kleinman}, {\em {On an Iterative Technique for Riccati Equation
  Computations}}, {IEEE} {T}ransactions on {A}utomatic {C}ontrol, 13 (1968),
  pp.~114--115.

\bibitem{Lancaster.Rodman.95}
{\sc P.~Lancaster and L.~Rodman}, {\em Algebraic {R}iccati equations}, Oxford
  Univ. Press, 1995.

\bibitem{Lin.Simoncini.13a}
{\sc Y.~Lin and V.~Simoncini}, {\em Minimal residual methods for large scale
  {L}yapunov equations}, Applied Num. Math., 72 (2013), pp.~52--71.

\bibitem{Lin.Simoncini.15}
{\sc Y.~Lin and V.~Simoncini}, {\em A new subspace iteration method for the
  algebraic {R}iccati equation}, Numerical Linear Algebra w/Appl., 22 (2015),
  pp.~26--47.

\bibitem{matlab7}
{\sc The MathWorks, Inc.}, {\em M{A}{T}{L}{A}{B} 7}, r2013b~ed., 2013.

\bibitem{Olsson2006}
{\sc K.~H.~A. Olsson and A.~Ruhe}, {\em Rational {Krylov} for eigenvalue
  computation and model order reduction}, {BIT Numerical Mathematics}, 46
  (2006), pp.~99--111.

\bibitem{Penzl2000b}
{\sc T.~Penzl}, {\em {A cyclic low-rank {Smith} method for large sparse
  {Lyapunov} equations}}, SIAM J. Sci. Comput., 21 (2000), pp.~1401--1418.

\bibitem{Ruhe1984}
{\sc A.~Ruhe}, {\em Rational {Krylov} sequence methods for eigenvalue
  computation}, Lin. Alg. Appl., 58 (1984), pp.~391--405.

\bibitem{Saad1990a}
{\sc Y.~Saad}, {\em Numerical solution of large {Lyapunov} equations}, in
  Signal Processing, Scattering, Operator Theory, and Numerical Methods.
  Proceedings of the international symposium MTNS-89, vol III, M.~A. Kaashoek,
  J.~H. van Schuppen, and A.~C. Ran, eds., Boston, 1990, Birkhauser,
  pp.~503--511.

\bibitem{Schilders2008}
{\sc W.~H.~A. Schilders, H.~A. van~der Vorst, and J.~Rommes}, {\em {Model Order
  Reduction: Theory, Research Aspects and Applications}}, Springer-Verlag,
  Berlin/Heidelberg, 2008.

\bibitem{Simoncini.survey13}
{\sc V.~Simoncini}, {\em Computational methods for linear matrix equations},
  tech. rep., Alma Mater Studiorum - Universit{\`a} di Bologna, 2013.
\newblock SIAM Review, Sept 2016.

\bibitem{Simoncini.16}
\leavevmode\vrule height 2pt depth -1.6pt width 23pt, {\em On the extended
  {K}rylov subspace method for the algebraic {R}iccati equation}, January 2016.
\newblock In preparation.

\bibitem{Simoncini.Szyld.Monsalve.13}
{\sc V.~Simoncini, D.~B. Szyld, and M.~Monsalve}, {\em On two numerical methods
  for the solution of large-scale algebraic {R}iccati equations}, IMA Journal
  of Numerical Analysis, 34 (2014), pp.~904--920.

\bibitem{Stewart.Sun.90}
{\sc G.~W. Stewart and J.-G. Sun}, {\em Matrix {P}erturbation {T}heory},
  Academic Press, 1990.

\end{thebibliography}

\end{document}